\documentclass[12pt]{amsart}
\usepackage{amssymb}
\usepackage{amsfonts}
\usepackage{amssymb,latexsym}
\usepackage{enumerate}
\usepackage{mathrsfs}
\usepackage{url} \usepackage{color}\makeatletter
\@namedef{subjclassname@2010}{%
  \textup{2010} Mathematics Subject Classification}
\makeatother

  [2002/01/19 v2.2g %
    AMS font definitions%
  ]
\DeclareFontFamily{U}{euf}{}
\DeclareFontShape{U}{euf}{m}{n}{%
  <5><6><7><8><9>gen*eufm%
  <10><10.95><12><14.4><17.28><20.74><24.88>eufm10%
  }{}
\DeclareFontShape{U}{euf}{b}{n}{%
  <5><6><7><8><9>gen*eufb%
  <10><10.95><12><14.4><17.28><20.74><24.88>eufb10%
  }{}

  [2002/01/19 v2.2g %
    AMS font definitions%
  ]
\DeclareFontFamily{U}{msb}{}
\DeclareFontShape{U}{msb}{m}{n}{%
  <5><6><7><8><9>gen*msbm%
  <10><10.95><12><14.4><17.28><20.74><24.88>msbm10%
  }{}

  [2002/01/19 v2.2g %
    AMS font definitions%
  ]
\DeclareFontFamily{U}{msa}{}
\DeclareFontShape{U}{msa}{m}{n}{%
  <5><6><7><8><9>gen*msam%
  <10><10.95><12><14.4><17.28><20.74><24.88>msam10%
  }{}

\newtheorem{thm}{Theorem}[section]

\newtheorem{cor}[thm]{Corollary}
\newtheorem{lem}[thm]{Lemma}

\newtheorem{rmk}[thm]{Remark}

\newtheorem{case}{Case}

\theoremstyle{definition}

\theoremstyle{remark}

\numberwithin{equation}{section} 

\begin{document}

\title[Menon-Sury's identity]{On Menon-Sury's identity with several Dirichlet characters}

\author{Man Chen}
\address{Department of Mathematics, South China University of Technology, Guangzhou 510640, China}
\email{13798043026@163.com}

\author{Su Hu}
\address{Department of Mathematics, South China University of Technology, Guangzhou 510640, China} \email{mahusu@scut.edu.cn}

\author{Yan Li*}
\address{Department of Applied Mathematics, China Agricultural
University, Beijing 100083, China} \email{liyan\_00@cau.edu.cn}

 \thanks{*Corresponding author}

\begin{abstract}
The Menon-Sury's identity is as follows:
\begin{equation*}
\sum_{\substack{1 \leq a, b_1, b_2, \ldots, b_r \leq n\\\mathrm{gcd}(a,n)=1}} \mathrm{gcd}(a-1,b_1, b_2, \ldots, b_r,n)=\varphi(n) \sigma_r(n),
\end{equation*}
where  $\varphi$ is Euler's totient function and $\sigma_r(n)=\sum_{d\mid n}{d^r}$.
Recently, Li, Hu and Kim \cite{L-K} extended  the above identity to a multi-variable case with a Dirichlet character, that is, they proved
 \begin{equation*}
\sum_{\substack{a\in\Bbb Z_n^\ast \\ b_1, \ldots, b_r\in\Bbb Z_n}} \mathrm{gcd}(a-1,b_1, b_2, \ldots, b_r,n)\chi(a)=\varphi(n)\sigma_r{\left(\frac{n}{d}\right)},
\end{equation*}
where $\chi$ is a Dirichlet character modulo $n$ and
$d$ is the conductor of $\chi$.

In this paper, we explicitly compute the sum
\begin{equation*}\sum_{\substack{a_1, \ldots, a_s\in\Bbb Z_n^\ast \\ b_1, ..., b_r\in\Bbb Z_n}}\gcd(a_1-1, \ldots, a_s-1,b_1, \ldots, b_r, n)\chi_{1}(a_1) \cdots \chi_{s}(a_s).\end{equation*}
where $\chi_{i} (1\leq i\leq s)$ are Dirichlet characters mod $n$ with conductor $d_i$. A special but common case of our main result reads like this :
\begin{equation*}\sum_{\substack{a_1, \ldots, a_s\in\Bbb Z_n^\ast \\ b_1, ..., b_r\in\Bbb Z_n}}\gcd(a_1-1, \ldots, a_s-1,b_1, \ldots, b_r, n)\chi_{1}(a_1) \cdots \chi_{s}(a_s)=\varphi(n)\sigma_{s+r-1}\left(\frac{n}{d}\right)\end{equation*}
if $d$ and $n$ have exactly the same prime factors, where $d={\rm lcm}(d_1,\ldots,d_s)$ is the least common multiple of $d_1,\ldots,d_s$.
Our result generalizes the above Menon-Sury's identity and Li-Hu-Kim's identity.\end{abstract}

\maketitle

\section{Introduction}
In 1965, P. K. Menon \cite{Menon} found the following beautiful identity,
 \begin{equation}\label{oldbegin1}
\sum_{\substack{a\in\Bbb Z_n^\ast
}}\gcd(a -1,n)=\varphi(n) \sigma_{0} (n),
\end{equation}
where $n$ is a positive integer, $\Bbb Z_n^\ast$ is the group of units of the ring $\Bbb Z_n=\Bbb Z/n\Bbb Z$, $\gcd(\ ,\ )$ represents the greatest common divisor,  $\varphi$
is the Euler's totient function and $\sigma_{r} (n) =\sum_{d|n } d^{r}$ is the divisor function.

%Many authors gave various generalizations of Menon's identity in different directions.
In 2009,  Sury \cite{Sury} generalized this identity in the following way
\begin{equation}\label{Sury}
\sum_{\substack{a\in\Bbb Z_n^\ast \\ b_1, \ldots, b_r\in\Bbb Z_n}} \mathrm{gcd}(a-1,b_1, b_2, \ldots, b_r,n)=\varphi(n) \sigma_r(n).
\end{equation}
%where $\sigma_r(n)=\sum_{d\mid n}{d^r}$.

The above Menon-Sury's identity has been generalized in several different directions.
First, it can be extended to arithmetic functions. This direction was considered by T\'oth  who generalized Menon's identity to sums representing arithmetical functions of several variables \cite[Theorems 1, 2]{Toth1}. His results involve several extensions of Menon's identity to multi-variable situations (e.g. \cite[Eq. (8)]{Toth1}).  As an application, he also presented a formula for the number of cyclic subgroups
of the direct product of several cyclic groups of arbitrary orders (see \cite[Theorem 3]{Toth1}).

Menon-Sury's identity can also be extended to residually finite Dedekind domains. It is well-known that the integer rings of number fields   and the integral closure of $\mathbb{F}_{q}[t]$ in the field extension $K/\mathbb{F}_{q}(t)$ with $K$ being an algebraic function field are all residually finite Dedekind domains. This direction was first done by Miguel in 2014 by using Burnside's lemma and the theory of commutative rings (see \cite{Mig1, Mig2}). Then  a further extension was made by
Li and Kim~\cite{LK1} who extended Miguel's result to the case with many tuples of group of units.  For the case $\mathbb{Z}$, their results  read
\begin{align}\label{LKD}
\sum_{\substack{a_1, \cdots, a_s\in\Bbb Z_n^\ast \\ b_1, ..., b_r\in\Bbb Z_n}} \mathrm{gcd}(a_1-1, \cdots, a_s-1,b_1, ..., b_r, n) \\
=\varphi(n) \prod_{i=1}^w(\varphi(p_i^{m_i})^{s-1}p_i^{m_ir}-p_i^{m_i(s+r-1)}+\sigma_{s+r-1}(p_i^{m_i}))\nonumber,
\end{align}
where $n=p_1^{m_1}\cdots p_w^{m_w}$ is the prime factorization of $n$.

Menon-Sury's identity can also be extended to subgroups of general linear group $\textrm{GL}_{r}(\mathbb{Z}_{n})$.
This direction was subsequently considered by T$\breve{a}$rn$\breve{a}$uceanu  \cite{tarn} for groups of upper triangular
matrices in  $\textrm{GL}_{r}(\mathbb{Z}_{n})$ who solved an open problem raised in ~\cite{Sury}. Li and Kim~\cite{LK2} further extended Menon-Sury's identity
  to unipotent groups, Heisenberg groups and extended Heisenberg groups by using  %Cauchy-Frobenius-
  Burnside's lemma for group actions of matrix multiplication on column vectors over $\mathbb{Z}_{n}$.

In 2017, Zhao and Cao \cite[Theorem 1.2]{Z-C} obtained a Menon-type identity with a Dirichlet character. In fact, they showed that
\begin{equation} \label{Cao}
\sum_{\substack{a\in\Bbb Z_n^\ast}} \mathrm{gcd}(a-1,n)\chi(a)=\varphi(n)\tau{\left(\frac{n}{d}\right)},
\end{equation}
where  $\chi$ is a Dirichlet character mod $n$ with conductor $d$.
Then T\'oth \cite{Toth2}  extended the above identity by considering even functions (mod $n$) from an alternative approach. As an application, he also obtained certain related formulas concerning Ramanujan sums.
Recently, Li, Hu and Kim \cite{L-K} extended (\ref{Cao}) to a multi-variable case, that is, they proved
 \begin{equation}\label{char1}
\sum_{\substack{a\in\Bbb Z_n^\ast \\ b_1, ..., b_r\in\Bbb Z_n}} \mathrm{gcd}(a-1,b_1, b_2, \ldots, b_r,n)\chi(a)=\varphi(n)\sigma_r{\left(\frac{n}{d}\right)}.
\end{equation}
%where $\chi$ is a Dirichlet character modulo $n$ and
%$d$ is the conductor of $\chi$.
For other extensions of Menon's identity with additive characters and multiplicative characters, see \cite{LK2} and \cite{LHK3}.
%It needs to remark that Li and Kim \cite{LK2} also showed a   Menon-type identity with additive characters in another paper published recently.

In this paper, generalizing Li, Hu and Kim's result ~(\ref{char1}), we consider the Menon-type identity involving several Dirichlet characters, that is, we evaluate the sum
\begin{equation}\begin{aligned}\label{chars}&S_{\chi_{1},\chi_{2},\ldots,\chi_{s}}(n,r)\\&=\sum_{\substack{a_1, \ldots, a_s\in\Bbb Z_n^\ast \\ b_1, ..., b_r\in\Bbb Z_n}}\gcd(a_1-1, \ldots, a_s-1,b_1, ..., b_r, n)\chi_{1}(a_1) \cdots \chi_{s}(a_s),\end{aligned}\end{equation}
where $\chi_{i} (1\leq i\leq s)$ are Dirichlet characters mod $n$ with conductor $d_i$. First, we explicit compute $S_{\chi_{1},\chi_{2},\ldots,\chi_{s}}(n,r)$ in the
assumption that $n$ is a prime power (see Theorem \ref{thhh3}). Then by using the Chinese remainder theorem, we pass to the general case (see Theorem~\ref{mainresult} and Remark~\ref{Remark 3.3}).

\section{Prime power case}
Throughout  this section, we  assume $n=p^m$, where $p$ is a prime number and $m$ is a positive integer.
 Let $\chi_{i}~ (1\leq i\leq s)$ be  Dirichlet characters modulo $n$ with conductor $d_i$.
 Since $d_i\mid n$, we have $d_i=p^{t_i}$, where $0\leq t_i\leq m$.\

As in \cite[p.46]{LK},
we shall introduce flirtations for %the additive group $\Bbb Z_n$ and
the multiplicative group $\Bbb Z_n^\ast$. %, respectively.
 As $n=p^m$ is a prime power,
% the whole subgroups of $\Bbb Z_n$ form a chain:
%\begin{equation}\nonumber
%0=p^m\Bbb Z_n\subset p^{m-1}\Bbb Z_n\subset ...\subset p\Bbb Z_n\subset \Bbb Z_n .
%\end{equation}
$\Bbb Z_n^\ast$ has a filtration consisting of subgroups:
\begin{equation}\nonumber
1=1+p^m\Bbb Z_n\subset 1+p^{m-1}\Bbb Z_n\subset ...\subset 1+p\Bbb Z_n\subset \Bbb Z_n^\ast.
\end{equation}
For simplicity of the proof, we introduce the following notations.
\begin{equation}\label{unit}
\begin{aligned}
%R_i&=p^i\Bbb Z_n \ \ {\rm  with} \ \ 0\leq i\leq m\ \  {\rm  and}\ \  R_{m+1}=\varnothing,\\
U_0&=\Bbb Z_n^\ast, U_j=1+p^j\Bbb Z_n\ \ {\rm  with} \ \ 1\leq j\leq m\ \  {\rm  and}\ \  U_{m+1}=\varnothing,\\
%S_i&=R_i-R_{i+1} \ \ {\rm  with} \ \  0\leq i\leq m,\\
V_j&=U_j-U_{j+1} \ \ {\rm  with} \ \  0\leq j\leq m.
\end{aligned}
\end{equation}
Clearly, %$\Bbb Z_n=\bigcup \limits_{i=0}^m S_i$ and
$\Bbb Z_n^\ast=\bigcup \limits_{j=0}^m V_j$ with disjoint union.
 Also, we have
\begin{align}\nonumber
\# U_0=p^m-p^{m-1}, \   \# U_{m+1}=0 \   {\rm  and}\
\# U_j=p^{m-j}  \ {\rm  with}\  1\leq j\leq m,\nonumber
\end{align}
where \# denotes the cardinality of sets.\

% Now, we begin to compute $S_{\chi}(p^m,k)$.
Since
\begin{equation*}%\label{LiAdd3}
\begin{split}
%&\ S_{\chi_{1},\chi_{2},\ldots,\chi_{s}}(p^m,r)\\
&\sum_{\substack{a_1, \ldots, a_s\in\Bbb Z_{p^m}^\ast \\ b_1, ..., b_r\in\Bbb Z_{p^m}}}
 \gcd(a_1-1, \ldots, a_s-1,b_1, ..., b_r, p^m)\chi_{1}(a_1) \cdots \chi_{s}(a_s)\\
=&\sum_{k=0}^m \sum_{\substack{\gcd (b_1, ..., b_r, p^m)=p^k \\ b_1, ..., b_r\in\Bbb Z_{p^m}}} \sum_{a_1, \ldots, a_s\in\Bbb Z_{p^m}^\ast} \gcd (a_1-1, \ldots, a_s-1, p^k) \chi_{1}(a_1) \cdots \chi_{s}(a_s),
\end{split}
\end{equation*}
we get
\begin{equation}\label{LiAdd3}
\begin{split}
&\ S_{\chi_{1},\chi_{2},\ldots,\chi_{s}}(p^m,r)\\
=&\sum_{k=0}^m\left(\sum_{a_1, \ldots, a_s\in\Bbb Z_{p^m}^\ast} \gcd (a_1-1, \ldots, a_s-1, p^k) \chi_{1}(a_1) \cdots \chi_{s}(a_s)\right)\\&\quad\times \left(\sum_{\substack{\gcd (b_1, ..., b_r, p^m)=p^k \\ b_1, ..., b_r\in\Bbb Z_{p^m}}}1\right).
\end{split}
\end{equation}
Therefore, we need to compute
$$\sum_{a_1, \ldots, a_s\in\Bbb Z_{p^m}^\ast} \gcd (a_1-1, \ldots, a_s-1, p^k) \chi_{1}(a_1) \cdots \chi_{s}(a_s)\ \  {\rm  and}\  \sum_{\substack{\gcd (b_1, ..., b_r, p^m)=p^k \\ b_1, ..., b_r\in\Bbb Z_{p^m}}}1.$$
 These will be done in Lemmas \ref{2.2} and \ref{thm3}, respectively. We may first need the following two lemmas (Lemmas \ref{1} and \ref{thm2}).
\begin{lem}[{Li, Hu and Kim, \cite[Lemma 2.1]{L-K}}]\label{1}
Let $n=p^m$ and $\chi$ be a Dirichlet character modulo $n$ with conductor $p^t$, where $0\leq t\leq m$. Then, for $0\leq j\leq m$, we have
\begin{equation}\nonumber
\sum_{a\in U_j} \chi(a)=\left\{
						\begin{aligned}
						&\#U_j, \ \ if\ \  j=t,t+1,..., m.\\
						&0,\ \  otherwise.
						\end{aligned}
				    \right.
\end{equation}
\end{lem}
\begin{lem}\label{thm2}
Let $n=p^m$ and $\chi_{i}~(1\leq i \leq s)$ be  Dirichlet characters modulo $n$ with conductors $d_i=p^{t_i}$, where $0\leq t_i\leq m$. Let $u=\max\{t_1,..., t_s\}$.Then, for $0\leq j\leq m$, we have
\begin{equation}\nonumber
\sum_{(a_1, \ldots, a_s)\in (U_j)^s} \chi_{1}(a_1) \cdots \chi_{s}(a_s)=\left\{
						\begin{aligned}
						&(\#U_j)^s, \ \ \textrm{if}\ \  j=u,u+1,..., m.\\
						&0,\ \  \textrm{otherwise}.
						\end{aligned}
				    \right.
\end{equation}
\end{lem}
\begin{proof}
By Lemma \ref{1}, we have
\begin{align*}\nonumber
&\sum_{(a_1, \ldots, a_s)\in (U_j)^s} \chi_{1}(a_1) \cdots \chi_{s}(a_s)\\
&=\sum_{a_1 \in U_j} \chi_{1}(a_1)\cdots \sum_{a_s \in U_j} \chi_{s}(a_s)\\
&=\# U_j[ j\geq t_1 ]\cdots\# U_j [j\geq t_s]\\
&=(\# U_j)^s[j\geq \max(t_1,\ldots,t_s)],
\end{align*}
%&={\left\{
%						\begin{aligned}
%						&\# U_j, \ \ \textrm{if}\ \  j=t_1,t_1+1,..., m.\\
%						&0,\ \  \textrm{otherwise}.
%						\end{aligned}
%				    \right.}
%\cdots
%{\left\{
%						\begin{aligned}
%						&\# U_j, \ \ \textrm{if}\ \  j=t_s,t_s+1,..., m.\\
%						&0,\ \  \textrm{otherwise}.
%						\end{aligned}
%				    \right.}\\
%&=\left\{
%						\begin{aligned}
%						&(\#U_j)^s, \ \ \textrm{if}\ \  j=\max\{t_1,..., t_s\},\max\{t_1,..., %t_s\}+1,..., m.\\
%						&0,\ \  \textrm{otherwise}.
%						\end{aligned}
%				    \right.
where [\ \ ] is the Iverson bracket, i.e.
\begin{equation*}[P]={\left\{
						\begin{aligned}
						&1, \ \ \textrm{if condition}\ P\ \textrm{holds;}\\
						&0,\ \  \textrm{otherwise}.
						\end{aligned}
				    \right.}
\end{equation*}
\end{proof}
\begin{lem}\label{2.2}
Let $n=p^m$ and $\chi_{i}~(1\leq i\leq s)$ be  Dirichlet characters modulo $n$ with conductors $p^{t_i}$, where $0\leq t_i\leq m$.
Let $k$ be an integer such that $0\leq k\leq m$. Let $u=\max\{t_1,t_2,\cdots, t_s\}$. Then we have
\begin{equation*}%\label{CHL0}
\begin{split}
&\sum_{a_1, \ldots, a_s\in\Bbb Z_n^\ast} \gcd(a_1-1, \ldots, a_s-1, p^k) \chi_{1}(a_1) \cdots \chi_{s}(a_s)\\
=&\left\{
						\begin{aligned}
						&0, \  &\textrm{if} \ k<u; \\
                        &\varphi(p^m)p^{(m-k)(s-1)}\sigma_{s-1}(p^{k-u}),& \textrm{if}\ k\geq u>0;\\
                                                %&\sum_{j=u}^{k}(p^j-p^{j-1})(\#U_j)^s,& \textrm{if}\ k\geq u>0;\\
						&\varphi(p^m)^s-\varphi(p^m)p^{m(s-1)}+\varphi(p^m)p^{(m-k)(s-1)}\sigma_{s-1}(p^{k}), &\textrm{if}\ u=0.
%&(\# U_0)^s+\sum_{j=1}^k(p^j-p^{j-1})(\# U_j)^s, &\textrm{otherwise}.
						\end{aligned}
				    \right.
\end{split}
\end{equation*}
\end{lem}
\begin{proof}
By \eqref{unit} and direct computation, we have
\begin{align*}
&\sum_{a_1, \ldots, a_s\in\Bbb Z_n^\ast} \gcd(a_1-1, \ldots, a_s-1, p^k) \chi_{1}(a_1) \cdots \chi_{s}(a_s)\\
=&\sum_{j=0}^m \sum_{(a_1, \ldots, a_s)\in(U_ j)^s-(U_{j+1})^s}\gcd(a_1-1, \ldots, a_s-1, p^k)\chi_{1}(a_1) \cdots \chi_{s}(a_s)\\
=&\sum_{j=0}^{k-1} \sum_{(a_1, \ldots, a_s)\in(U_ j)^s-(U_{j+1})^s} p^j \chi_{1}(a_1) \cdots \chi_{s}(a_s)+\sum_{j=k}^m \sum_{(a_1, \ldots, a_s)\in(U_ j)^s-(U_{j+1})^s}p^k \chi_{1}(a_1) \cdots \chi_{s}(a_s)\\
=&\sum_{j=0}^{k-1}p^j\left(\ \sum_{(a_1, \ldots, a_s)\in(U_ j)^s}\chi_{1}(a_1) \cdots \chi_{s}(a_s)- \sum_{(a_1, \ldots, a_s)\in(U_{j+1})^s}\chi_{1}(a_1) \cdots \chi_{s}(a_s)\right)\\
&+p^k\sum_{(a_1, \ldots, a_s)\in(U_k)^s}\chi_{1}(a_1) \cdots\chi_{s}(a_s)\\
=&\sum_{j=0}^{k} p^j \sum_{(a_1, \ldots, a_s)\in(U_j)^s} \chi_{1}(a_1) \cdots \chi_{s}(a_s)-\sum_{j=1}^k p^{j-1} \sum_{(a_1, \ldots, a_s)\in(U_ j)^s} \chi_{1}(a_1) \cdots \chi_{s}(a_s).
%=&\sum_{(a_1, \ldots, a_s)\in(U_0)^s}\chi_{1}(a_1) \cdots \chi_{s}(a_s)+\sum_{j=1}^{k}(p^j-p^{j-1})\sum_{(a_1, \ldots, a_s)\in(U_j)^s}\chi_{1}(a_1) \cdots \chi_{s}(a_s).\nonumber%\\
%=&\sum_{a_1\in U_0}\chi_{1}(a_1)\cdots \sum_{a_s\in U_0}\chi_{s}(a_s)+\sum_{j=1}^k(p^j-p^{j-1})\left(\sum_{a_1\in U_j}\chi_{1}(a_1)\cdots \sum_{a_s\in U_j}\chi_{s}(a_s)\right).
\end{align*}
Collecting the similar items, we get
\begin{equation}\label{Rainy1}
\begin{split}
&\sum_{a_1, \ldots, a_s\in\Bbb Z_n^\ast} \gcd(a_1-1, \ldots, a_s-1, p^k) \chi_{1}(a_1) \cdots \chi_{s}(a_s)\\
=&\sum_{(a_1, \ldots, a_s)\in(U_0)^s}\chi_{1}(a_1) \cdots \chi_{s}(a_s)+\sum_{j=1}^{k}(p^j-p^{j-1})\sum_{(a_1, \ldots, a_s)\in(U_j)^s}\chi_{1}(a_1) \cdots \chi_{s}(a_s).
\end{split}
\end{equation}
Now we need to calculate the above sum case by case.
\begin{case} $k<\max\{t_1,t_2,\cdots, t_s\}$.\\ %and  $t_1,t_2,\cdots, t_s\geq 1$
%By  Lemma \ref{thm2}, we have

\rm{Substituting Lemma \ref{thm2} into \eqref{Rainy1}, we have}
\begin{align*}
&\sum_{a_1, \ldots, a_s\in\Bbb Z_n^\ast} \gcd(a_1-1, \ldots, a_s-1, p^k) \chi_{1}(a_1) \cdots \chi_{s}(a_s)\\
%&=\sum_{a_1\in U_0}\chi_{1}(a_1)\cdots \sum_{a_s\in U_0}\chi_{s}(a_s)+\sum_{j=1}^k(p^j-p^{j-1})\left(\sum_{a_1\in U_j}\chi_{1}(a_1)\cdots \sum_{a_s\in U_j}\chi_{s}(a_s)\right)\\
%&=0+\sum_{j=1}^k(p^j-p^{j-1})\sum_{(a_1, \ldots, a_s)\in(U_j)^s}\chi_{1}(a_1) \cdots \chi_{s}(a_s)\nonumber\\\\
&=0+\sum_{j=1}^k(p^j-p^{j-1})\cdot 0=0.
\end{align*}
\end{case}

\begin{case}$k\geq\max\{t_1,t_2,\cdots, t_s\}>0$.\\
% There are some $t_i$=0, that is, $t_i=0$ if and only if $i=\{i_1,i_2,\cdots, i_\lambda\}\subsetneq \{1,2,\cdots, s\}$.

%In this case, according to Lemma \ref{1}, we have
%\begin{align}
%\sum_{a_{i_v}\in U_j} \chi_{d_{i_v}}(a_{i_v})=\# U_j ~~(v=1,2,\cdots,\lambda),\\
%\sum_{a_i\in U_j} \chi_{i}(a_i)=\left\{
%                       \begin{aligned}
%						&\#U_j, \ \ \textrm{if}\ \  j\geq t_i,\\
%						&0,\ \ \textrm{if}\ \  j<t_i.
%\end{aligned}
%				    \right.
%\end{align}
%where $i\notin \{i_1,i_2,\cdots, i_\lambda\}$.
\rm{Substituting Lemma \ref{thm2} into \eqref{Rainy1} and from $u=\max\{t_1,\cdots, t_s\}$, we get}
\begin{align*}
&\sum_{a_1, \ldots, a_s\in\Bbb Z_{p^m}^\ast} \gcd(a_1-1, \ldots, a_s-1, p^k) \chi_{1}(a_1) \cdots \chi_{s}(a_s)\\
%&=\sum_{a_1\in U_0}\chi_{1}(a_1)\cdots \sum_{a_s\in U_0}\chi_{s}(a_s)+\sum_{j=1}^k(p^j-p^{j-1})\left(\sum_{a_1\in U_j}\chi_{1}(a_1)\cdots \sum_{a_s\in U_j}\chi_{s}(a_s)\right)\\
%&=0+\sum_{j=1}^{u-1}(p^j-p^{j-1})\left(\sum_{a_1\in U_j}\chi_{1}(a_1)\cdots \sum_{a_s\in U_j}\chi_{s}(a_s)\right)
%\\&\quad+\sum_{j=u}^{k}(p^j-p^{j-1})\left(\sum_{a_1\in U_j}\chi_{1}(a_1)\cdots \sum_{a_s\in U_j}\chi_{s}(a_s)\right)\\
&=\sum_{j=u}^{k}(p^j-p^{j-1})(\#U_j)^s\\
%&=\sum_{j=u}^{k}(p^j-p^{j-1})p^{(m-j)s}\\
&=(p^m-p^{m-1})\sum_{j=u}^{k}p^{(m-j)(s-1)}\\
&=\varphi(p^m)p^{(m-k)(s-1)}\sigma_{s-1}(p^{k-u}).
%&=\left(1-\frac{1}{p}\right)p^{ms}\frac{(p^{1-s})^u(1-p^{(1-s)(k-u+1)})}{1-p^{1-s}}.
\end{align*}
\end{case}
\begin{case}$\max\{t_1,\cdots, t_s\}=0$, i.e. $\chi_{i}~(1\leq i\leq s)$ are trivial characters.\

\rm{The same argument as in \textbf{Case 2} shows that}
\begin{align*}
&\sum_{a_1, \ldots, a_s\in\Bbb Z_{p^m}^\ast} \gcd(a_1-1, \ldots, a_s-1, p^k) \chi_{1}(a_1) \cdots \chi_{s}(a_s)\\
%&=\sum_{a_1\in U_0}\chi_{1}(a_1)\cdots \sum_{a_s\in U_0}\chi_{s}(a_s)+\sum_{j=1}^k(p^j-p^{j-1})\left(\sum_{a_1\in U_j}\chi_{1}(a_1)\cdots \sum_{a_s\in U_j}\chi_{s}(a_s)\right)\\
&=(\# U_0)^s+\sum_{j=1}^k(p^j-p^{j-1})(\# U_j)^s\\
%&=(p^m-p^{m-1})^s+\sum_{j=1}^k(p^j-p^{j-1})p^{(m-j)s}\\
&=(p^m-p^{m-1})^s-(1-p^{-1})p^{ms}+\sum_{j=0}^k(p^j-p^{j-1})p^{(m-j)s}\\
&=\varphi(p^m)^s-\varphi(p^m)p^{m(s-1)}+\varphi(p^m)p^{(m-k)(s-1)}\sigma_{s-1}(p^{k}).
\end{align*}
\end{case}
\end{proof}
\begin{lem}[{Li, Hu and Kim \cite[Lemma 2.4]{L-K}}]\label{thm3}
Let $n=p^m$ be a prime power and $k\geq 0$ be an integer.
Assume $r\geq 0$ is an integer.
 Then
\begin{equation}\nonumber
\sum_{\substack{b_1, ..., b_r\in \Bbb Z_{p^m} \\ \gcd(b_1, ..., b_r,p^m)=p^k}} 1=\left\{
						\begin{aligned}
						p^{(m-k)r}-&p^{(m-k-1)r}, \ &if\   k< m,\\
						&1,\  &if\  k=m.
						\end{aligned}
				    \right.
\end{equation}
\end{lem}
Now we  are at the position to evaluate the sum $S_{\chi_{1},\chi_{2},\cdots,\chi_{s}}(p^m,r)$.%, where $\chi_{i}~(1\leq i\leq s)$ are Dirichlet characters modulo prime powers.
 \begin{thm}\label{thhh3}
 Assume $r\geq 0$ and $s>0$ are integers. Let $n=p^m$ be a prime power and $\chi_{i}~(1\leq i\leq s)$ be   Dirichlet characters modulo $n$ with conductors  $d_i=p^{t_i}$.
  Let $u=\max\{t_1,\cdots, t_s\}$. Then we have the following identity
\begin{equation}\begin{split} &S_{\chi_{1},\chi_{2},\ldots,\chi_{s}}(n,r)\\=&
\left\{
						\begin{aligned}
						&\varphi(p^m)\sigma_{s+r-1}\left(p^{m-u}\right), \ &if\   u>0,\\
						&\varphi(p^m)(\varphi(p^m)^{s-1}p^{mr}+\sigma_{s+r-1}(p^{m})-p^{m(s+r-1)}),\  &if\  u=0.
						\end{aligned}
				    \right.
\end{split}
\end{equation}

 \end{thm}

\begin{proof}
Firstly, assume $u>0$.
 %By equation \eqref{LiAdd3}, we have
 %\begin{align}\label{LiAdd4}
%&S_{\chi_{1},\chi_{2},\ldots, \chi_{s}}(n,r)\\ \nonumber
%&=\sum_{k=0}^m\left(\sum_{a_1, \ldots, a_s\in\Bbb Z_{p^m}^\ast} \gcd (a_1-1, \ldots, a_s-1, p^k) \chi_{1}(a_1) \cdots \chi_{s}(a_s)\right) \left(\sum_{\substack{\gcd (b_1, ..., b_r, p^m)=p^k \\ b_1, ..., b_r\in\Bbb Z_{p^m}}}1\right).
%\end{align}

Substituting Lemma \ref{2.2} into \eqref{LiAdd3}, we get
\begin{equation}\label{LiAdd4}
\begin{split}
&S_{\chi_{1},\chi_{2},\ldots,\chi_{s}}(n,r)\\
=&\sum_{k=u}^m\varphi(p^m)p^{(m-k)(s-1)}\sigma_{s-1}(p^{k-u}) \left(\sum_{\substack{\gcd (b_1, ..., b_r, p^m)=p^k \\ b_1, ..., b_r\in\Bbb Z_{p^m}}}1\right)\\
=&\frac{\varphi(p^m)p^{(m-u)(s-1)}}{1-p^{1-s}}\sum_{k=u}^m\left(1-p^{(1-s)(k-u+1)}\right)
\left(\sum_{\substack{\gcd (b_1, ..., b_r, p^m)=p^k \\ b_1, ..., b_r\in\Bbb Z_{p^m}}}1\right).\end{split}
\end{equation}
The last equality is due to
$$\sigma_{s-1}(p^{k-u})=p^{(k-u)(s-1)}\frac{1-p^{(1-s)(k-u+1)}}{1-p^{1-s}}.
$$

From Lemma \ref{thm3},
\begin{equation}\label{LiAdd5}
\begin{split}
&\sum_{k=u}^m\left(1-p^{(1-s)(k-u+1)}\right)
\left(\sum_{\substack{\gcd (b_1, ..., b_r, p^m)=p^k \\ b_1, ..., b_r\in\Bbb Z_{p^m}}}1\right)\\
=&\sum_{k=u}^{m-1}\left(1-p^{(1-s)(k-u+1)}\right)
\left(p^{(m-k)r}-p^{(m-k-1)r}\right)+\left(1-p^{(1-s)(m-u+1)}\right)\\
%=& \sum_{k=u}^m\left(1-p^{(1-s)(k-u+1)}\right)p^{(m-k)r}-\sum_{k=u}^{m-1}\left(1-p^{(1-s)(k-u+1)}\right)p^{(m-k-1)r}\\
=& \sum_{k=u}^m\left(1-p^{(1-s)(k-u+1)}\right)p^{(m-k)r}-\sum_{k=u+1}^{m}\left(1-p^{(1-s)(k-u)}\right)p^{(m-k)r}\\
=& \left(1-p^{1-s}\right)p^{(m-u)r}+\sum_{k=u+1}^{m}\left(p^{(1-s)(k-u)}-p^{(1-s)(k-u+1)}\right)p^{(m-k)r}\\
=& \left(1-p^{1-s}\right)\sum_{k=u}^{m}p^{(1-s)(k-u)}p^{(m-k)r}\\
=& \left(1-p^{1-s}\right)\sum_{k=0}^{m-u}p^{(1-s)k}p^{(m-k-u)r}\\
=& \left(1-p^{1-s}\right)p^{(m-u)r}\sum_{k=0}^{m-u}p^{(1-s-r)k}.
%=&(1-p^{1-s})p^{(m-u)r}+\sum_{k=0}^{m-u-1}(p^{(1-s)(m-k-u)}-p^{(1-s)(m-k-u+1)})p^{kr}\\
%=&\sum_{k=0}^{m-u}(p^{(1-s)(m-k-u)}-p^{(1-s)(m-k-u+1)})p^{kr}\\
\end{split}
\end{equation}
Combining \eqref{LiAdd4} and \eqref{LiAdd5} together, we get
\begin{equation}\label{LiAdd6}
\begin{split}
&S_{\chi_{1},\chi_{2},\ldots,\chi_{s}}(n,r)\\
=&\varphi(p^m)p^{(m-u)(s+r-1)}\sum_{k=0}^{m-u}p^{(1-s-r)k}=\varphi(p^m)\sigma_{s+r-1}\left(p^{m-u}\right).
\end{split}
\end{equation}

Secondly, we treat the case $u=0$, i.e. all $\chi_{i}$-s are trivial characters.

Substituting Lemma \ref{2.2} into \eqref{LiAdd3}, we get
\begin{equation}\label{LiAdd7}
\begin{split}
&S_{\chi_{1},\chi_{2},\ldots,\chi_{s}}(n,r)\\
=&\sum_{k=0}^m\left(						\varphi(p^m)^s-\varphi(p^m)p^{m(s-1)}+\varphi(p^m)p^{(m-k)(s-1)}\sigma_{s-1}(p^{k})  \right)\\ &\times \left(\sum_{\substack{\gcd (b_1, ..., b_r, p^m)=p^k \\ b_1, ..., b_r\in\Bbb Z_{p^m}}}1\right)\\
=&\left(						\varphi(p^m)^s-\varphi(p^m)p^{m(s-1)}\right)\sum_{k=0}^m\left(\sum_{\substack{\gcd (b_1, ..., b_r, p^m)=p^k \\ b_1, ..., b_r\in\Bbb Z_{p^m}}}1\right)\\
&+\sum_{k=0}^m\varphi(p^m)p^{(m-k)(s-1)}\sigma_{s-1}(p^{k}) \left(\sum_{\substack{\gcd (b_1, ..., b_r, p^m)=p^k \\ b_1, ..., b_r\in\Bbb Z_{p^m}}}1\right).
%=&\frac{\varphi(p^m)p^{(m-u)(s-1)}}{1-p^{1-s}}\sum_{k=u}^m\left(1-p^{(1-s)(k-u+1)}\right)
%\left(\sum_{\substack{\gcd (b_1, ..., b_r, p^m)=p^k \\ b_1, ..., b_r\in\Bbb Z_{p^m}}}1\right).
\end{split}
\end{equation}
The same argument as in the case $u>0$ (i.e., letting $u=0$ in the second line of equation \eqref{LiAdd4}) shows that
\begin{equation}\label{LiAdd8}
  \begin{split}
  &\sum_{k=0}^m\varphi(p^m)p^{(m-k)(s-1)}\sigma_{s-1}(p^{k}) \left(\sum_{\substack{\gcd (b_1, ..., b_r, p^m)=p^k \\ b_1, ..., b_r\in\Bbb Z_{p^m}}}1\right)\\=&\varphi(p^m)\sigma_{s+r-1}\left(p^{m}\right).
  \end{split}
\end{equation}
Substituting \eqref{LiAdd8} and the equation
\begin{equation}
 \sum_{k=0}^m\left(\sum_{\substack{\gcd (b_1, ..., b_r, p^m)=p^k \\ b_1, ..., b_r\in\Bbb Z_{p^m}}}1\right)=p^{mr}
\end{equation}
into \eqref{LiAdd7}, we get the desired result.
\end{proof}

\section{The general case}
In this section, we shall consider the general case, that is, we evaluate $S_{\chi_{1},\chi_{2},\cdots,\chi_{s}}(n,r)$
if $\chi_{i}~(1\leq i\leq s)$ are Dirichlet characters modulo any positive integer $n$.
 First, from the Chinese remainder theorem, we show  $S_{\chi_{1},\chi_{2},\cdots,\chi_{s}}(n,r)$  is multiplicative with respect to $n$.
 Then using multiplicative property, we pass to the general case.

Let $n=n_1 n_2$ be the product of positive integers $n_1$ and $n_2$ such that
$\gcd(n_1,n_2)=1$. By the Chinese remainder theorem, we have the ring isomorphism:
$\Bbb Z_n \simeq \Bbb Z_{n_1}\oplus \Bbb Z_{n_2}$, which induces the multiplicative group isomorphism:
$\Bbb Z_n^\ast \simeq \Bbb Z_{n_1}^\ast\times \Bbb Z_{n_2}^\ast$.
Therefore, each Dirichlet character modulo $n$ can be uniquely written as $\chi=\chi^{(1)}\cdot\chi^{(2)}$,
 where $\chi$, $\chi^{(1)}$ and $\chi^{(2)}$ are Dirichlet characters modulo
  $n$, $n_1$ and $n_2$, respectively.
  Explicitly,
\begin{equation}\nonumber
\chi(c\  {\rm  mod}\  n)= \chi^{(1)}(c\ {\rm  mod}\  n_1)\cdot \chi^{(2)}(c\ {\rm  mod}\  n_2)
\end{equation}
for any integer $c$ such that $\gcd (c,n)=1.$\

To simplify notations, for $a_1, \ldots, a_s \in \Bbb Z_n$, we let $a'_1, \cdots a'_s\in \Bbb Z_{n_1}$ and $a''_1, \cdots, a''_s\in \Bbb Z_{n_2}$ denote the image of
 $a_1, \ldots, a_s $ in $\Bbb Z_{n_1}$ and $\Bbb Z_{n_2}$, respectively, i.e. $a_i'\equiv a_i\mod n_1$ and $a''_i\equiv a_i\mod n_2,$ for $ i=1, 2, \cdots, s$.
  Let $d, d^{(1)}$ and $d^{(2)}$ be the conductors of $\chi, \chi^{(1)}$ and $\chi^{(2)}$, respectively.
  It is well known that $d=d^{(1)} d^{(2)}$.

The following lemma shows that $S_{\chi_{1},\chi_{2},\ldots,\chi_{s}}(n,r)$ is multiplicative.
\begin{lem}\label{lem3}
Notations as above, we have
$$ S_{\chi_{1},\chi_{2},\ldots,\chi_{s}}(n,r)= S_{\chi_{1}^{(1)},\chi_{1}^{(1)},\ldots,\chi_{s}^{(1)}}(n_{1},r)\cdot  S_{\chi_{1}^{(2)},\chi_{2}^{(2)},\cdots,\chi_{s}^{(2)}}(n_{2},r).$$
\end{lem}
\begin{proof}  First, we check that
\begin{align*}
&\sum_{\substack{a_1, \ldots, a_s\in\Bbb Z_n^\ast \\ b_1, ..., b_r\in\Bbb Z_n}}
 \gcd(a_1-1, \ldots, a_s-1,b_1, ..., b_r, n)\chi_{1}(a_1) \cdots \chi_{s}(a_s)\\
=&\sum_{\substack{a_1, \ldots, a_s\in\Bbb Z_n^\ast \\ b_1, ..., b_r\in\Bbb Z_n}} \gcd(a_1-1, \ldots, a_s-1,b_1, ..., b_r, n_1) \gcd(a_1-1, \ldots, a_s-1,b_1, ..., b_r, n_2)\\
&\quad\times\chi_{1}^{(1)}(a_1) \cdots \chi_{s}^{(1)}(a_s)\chi_{1}^{(2)}(a_1) \cdots \chi_{s}^{(2)}(a_s) \\
=&\sum_{\substack{a'_1, \cdots, a'_s\in\Bbb Z_{n_1}^\ast \\ b'_1, ..., b'_r\in\Bbb Z_{n_1}}}\gcd(a'_1-1, \cdots, a'_s-1, b'_1, ..., b'_r, n_1)\chi_{1}^{(1)}(a'_1) \cdots \chi_{s}^{(1)}(a'_s)\\&\quad\times\sum_{\substack{a''_1, \cdots, a''_s\in\Bbb Z_{n_2}^\ast \\ b''_1, ..., b''_r\in\Bbb Z_{n_2}}}\gcd(a''_1-1, \cdots, a''_s-1, b''_1, ..., b''_r, n_2)\chi_{1}^{(2)}(a''_1) \cdots \chi_{s}^{(2)}(a''_s).
\end{align*}
The last equality follows from the Chinese remainder theorem.
Indeed, as $(a_1, \ldots, a_s, b_1,...,b_r)$ runs over ($\Bbb Z_n^\ast)^s\times(\Bbb Z_{n})^r, (a'_1, \cdots, a'_s, b'_1,..., b'_r, a''_1, \cdots, a''_s, b''_1,...,b''_r)$ runs over ($\Bbb Z_{n_1}^\ast)^s\times(\Bbb Z_{n_1})^r\times(\Bbb Z_{n_2}^\ast)^s\times(\Bbb Z_{n_2})^r$, too.
Therefore, we have
$$ S_{\chi_{1},\chi_{2},\ldots,\chi_{s}}(n,r)= S_{\chi_{1}^{(1)},\chi_{1}^{(1)},\ldots,\chi_{s}^{(1)}}(n_{1},r)\cdot  S_{\chi_{1}^{(2)},\chi_{2}^{(2)},\ldots,\chi_{s}^{(2)}}(n_{2},r).$$\end{proof}
\begin{thm}\label{mainresult}
 Assume $r\geq 0$, $s>0$ and $n>0$ are integers. Let $\chi_{i}~(1\leq i\leq s)$ be   Dirichlet characters modulo $n$ with conductors  $d_i$.
   Let $d={\rm lcm}(d_1,\ldots,d_s)$ be the least common multiple of $d_1,\ldots,d_s$. Let $n_0|n$ be the greatest integer such that $n_0$ and $d$ have the same prime factors.
   Then we have the following identity
\begin{equation}\label{general}\begin{split}&\sum_{\substack{a_1, \ldots, a_s\in\Bbb Z_n^\ast \\ b_1, ..., b_r\in\Bbb Z_n}}\gcd(a_1-1, \ldots, a_s-1,b_1, ..., b_r, n)\chi_{1}(a_1) \cdots \chi_{s}(a_s)\\
=&\varphi(n)\sigma_{s+r-1}\left(\frac{n_0}{d}\right)
\prod_{p^m||n/n_0}\left(\varphi(p^m)^{s-1}p^{mr}+\sigma_{s+r-1}(p^{m})-p^{m(s+r-1)}\right),\end{split}\end{equation}
where $p^m||n/n_0$ means $p^m$ exactly divides $n/n_0$.
 \end{thm}
\begin{proof}
%Let $n=p_1^{m_1}p_2^{m_2}\cdots p_w^{m_w}$ be the prime
Let $n=\prod p^{m_p}$ be the prime factorization of $n$.
Then, for each $i~(1\leq i\leq s)$, we have a decomposition %$\chi_{i}=\chi_{i}^{(1)}\chi_{i}^{(2)}\cdots\chi_{i}^{(w)}$,
$\chi_{i}=\prod_{p|n}\chi_{i}^{(p)}$,
where $\chi_{i}^{(p)}$ is a Dirichlet character modulo $p^{m_p}$ with conductor $d_{i}^{(p)}$.  %=p_{i}^{t_{i}^{(j)}},$ for $1\leq j\leq w$.
 It is easy to see that $d_{i}=\prod_{p|n}d_{i}^{(p)}$, %$d_{i}=d_{i}^{(1)}d_{i}^{(2)}\cdots d_{i}^{(u)}$
 for $1\leq i\leq s$. Let
 \begin{equation}\label{F1}d^{(p)}={\rm lcm}\left(d_{1}^{(p)},d_{2}^{(p)},\ldots,d_{s}^{(p)}\right)\end{equation}

 By \eqref{F1}, definition of $n_0$ and $d$, $d=\prod_{p|n_0}d^{(p)}$ and $\gcd(n_0,n/n_0)=1$. So for $p|n$, we have
 \begin{equation}
 d^{(p)}=1\Leftrightarrow p\nmid n_0.
 \end{equation}
 Theorefore, applying Theorem \ref{thhh3} to  $\chi_{1}^{(p)},\chi_{2}^{(p)},\ldots,\chi_{s}^{(p)}$, we have
 \begin{equation}\label{F2}
S_{\chi_{1}^{(p)},\chi_{2}^{(p)},\ldots,\chi_{s}^{(p)}}(p^{m_{p}},r)
=\varphi(p^{m_{p}})\sigma_{s+r-1}\left(\frac{p^{m_{p}}}{d^{(p)}}\right)
 \end{equation}
if $p|n_0$; and otherwise for $p|(n/n_0)$, we have
%where  $u_{j}=\max\{t_{1}^{(j)},t_{2}^{(j)}\cdots, t_{s}^{(j)}\}$.
 \begin{equation}\label{F3}\begin{split}&S_{\chi_{1}^{(p)},\chi_{2}^{(p)},\ldots,\chi_{s}^{(p)}}(p^{m_{p}},r)\\
 =&\varphi(p^{m_{p}})(\varphi(p^{m_p})^{s-1}p^{m_{p}r}+\sigma_{s+r-1}(p^{m_p})-p^{m_{p}(s+r-1)}).
 \end{split}
 \end{equation}
%Since the arithmetic functions $\varphi$ is multiplicative,
Applying Lemma \ref{lem3},  \eqref{F2} and \eqref{F3}, we have
\begin{equation}~\label{F4} \begin{aligned} &\qquad S_{\chi_{1},\chi_{2},\ldots,\chi_{s}}(n,r)\\
&=\prod_{p|n}S_{\chi_{1}^{(p)},\chi_{2}^{(p)},\ldots,\chi_{s}^{(p)}}(p^{m_{p}},r)
\\&=\prod_{p|n_0}\varphi(p^{m_{p}})\sigma_{s+r-1}\left(\frac{p^{m_{p}}}{d^{(p)}}\right)\\
%\\&=\prod_{p|n}\varphi(p^{m_p})\sigma_{s+r-1}\left(\frac{p_{j}^{m_{j}}}{p_{j}^{u_j}}\right)\\
&\times\prod_{p|n/n_0}\varphi(p^{m_{p}})\left(\varphi(p^{m_{p}})^{s-1}p^{m_{p}r}+\sigma_{s+r-1}(p^{m_{p}})-p^{m_{p}(s+r-1)}\right).
%&=\varphi(n)\prod_{j=1}^{{w}}\sigma_{s+r-1}\left(\frac{p_{j}^{m_{j}}}{p_{j}^{u_j}}\right),
\end{aligned}\end{equation}
Applying the multiplicative property of $\varphi$ and $\sigma_{s+r-1}$, and the equation $d=\prod_{p|n_0}d^{(p)}$ into \eqref{F4}, we obtain the identity in the general case.
\end{proof}
\begin{rmk}\label{Remark 3.3} If $s=1$ and $\chi_{1 }$ is a trivial character modulo $n$,  (\ref{general}) reduces to Sury's identity (\ref{Sury}). If  $d_1=d_2=\ldots d_s=1$ in (\ref{general}), that is, $\chi_{1},\chi_{2},\ldots,\chi_{s}$ are trivial characters, then  we reproduce Li and Kim's identity ({\ref{LKD}}). Letting $s=1$ and $r=1$ in (\ref{general}), we recover Zhao and Cao's identity (\ref{Cao}). Letting $s=1$ in~(\ref{general}), we recover Li,  Hu and  Kim's identity (\ref{char1}).   %Thus we provide another proof for this identity.
\end{rmk}

Finally, we mention that, in many cases, Theorem~\ref{mainresult} implies the following result.
\begin{cor}
 Assume $r\geq 0$, $s>0$ and $n>0$ are integers. Let $\chi_{i}~(1\leq i\leq s)$ be   Dirichlet characters modulo $n$ with conductors  $d_i$.
   Let $d={\rm lcm}(d_1,\ldots,d_s)$ be the least common multiple of $d_1,\ldots,d_s$. If $n$ and $d$ have exactly the same prime factors,  
   then we have the following identity
\begin{equation*}\begin{split}\sum_{\substack{a_1, \ldots, a_s\in\Bbb Z_n^\ast \\ b_1, ..., b_r\in\Bbb Z_n}}\gcd(a_1-1, \ldots, a_s-1,b_1, ..., b_r, n)\chi_{1}(a_1) \cdots \chi_{s}(a_s)
=\varphi(n)\sigma_{s+r-1}\left(\frac{n}{d}\right).
\end{split}\end{equation*}
 \end{cor}


\begin{thebibliography}{[20]}

\bibitem{H1} P. Haukkanen,  \textit{Menon's identity with respect to a generalized divisibility relation}, Aequationes Math. 70 (3) (2005) 240--246.

\bibitem{HW2} P. Haukkanen, J. Wang, \textit{High degree analogs of Menon's identity}, Indian J. Math. 39 (1) (1997) 37--42.
\bibitem{HW3} P. Haukkanen, J. Wang, \textit{A generalization of Menon's identity with respect to a set of polynomials}, Portugal. Math., 53 (3) (1996), 331--337.
%\bibitem{lang} S. Lang, \textit{Algebraic number theory}, GTM, Springer-Velag, 1994.

%\bibitem{Serre} Jean-Pierre Serre, \textit{On a theorem of Jordan}, Bull. Amer. Math. Soc. (N.S.)  40 (4), (2003), 429-440.

\bibitem{LK} Y. Li, D. Kim,  \textit{A Menon-type  identity with many tuples of  group of units in residually finite Dedekind domains}, J. Number Theory 175 (2017), 42-–50.


\bibitem{LK1} Y. Li, D. Kim,  \textit{Menon-type identities derived from actions of subgroups of general linear groups}, J.  Number Theory 179 (2017) 97--112.

\bibitem{L-K} Y. Li, X. Hu, D. Kim,  \textit{A generalization of Menon’s identity with Dirichlet characters},  Int. J. Number Theory, to appear, \url{ https://doi.org/10.1142/S1793042118501579}.

\bibitem{LK2} Y. Li, D. Kim,  \textit{Menon-type identities with additive characters}, J. Number Theory, to appear, \url{https://doi.org/10.1016/j.jnt.2018.04.023}.
\bibitem{LHK3} Y. Li, D. Kim,  \textit{A Menon-type Identity with Multiplicative and Additive Characters}, Taiwanese J. Math., to appear, \url{https://projecteuclid.org/euclid.twjm/1531382426}.

\bibitem{Menon} P. K. Menon, \textit{On the sum $\sum (a-1,n)[(a,n)=1]$},
J. Indian Math. Soc. (N.S.)  29 (1965),  155--163.

\bibitem{Mig1}  C. Miguel, \textit{Menon's identity in residually finite Dedekind domains}, J. Number Theory  137 (2014), 179--185.



\bibitem{Mig2}  C. Miguel, \textit{A Menon-type identity in residually finite Dedekind domains}, J. Number Theory  164 (2016), 43--51.


%\bibitem{Nark}
%W. Narkiewicz,  \textit{Elementary and analytic theory of algebraic numbers. Third edition}. Springer
 % Monographs in Mathematics. Springer-Verlag, Berlin, 2004.

\bibitem{Rama} V. S.  Ramaiah, \textit{Arithmetical sums in regular convolutions}, J. Reine Angew. Math. 303/304 (1978), 265--283.

%\bibitem{Ram}
%S. Ramanujan, \textit{On Certain Trigonometric Sums and their Applications in the Theory of Numbers}, Transactions of the Cambridge Philosophical Society, 22 (15), (1918) 259-–276
% (pp. 179-–199 of his Collected Papers).






%\bibitem{Richard} I. M. Richards, \textit{A remark on the number of cyclic subgroups of a finite group},
% Amer. Math.
%Monthly vol. 91 (1984), 571--572.


%\bibitem{Rosen} M. Rosen, \textit{Number Theory in Function Fields}, GTM, Springer-Velag, 2002.


%\bibitem{Rot} J. Rotman, \textit{A First Course in Abstract Algebra}, 2nd ed. Englewood Cliffs, NJ: Prentice-Hall, 2000.


%\bibitem{Siva1} R. Sivaramakrishnan, \textit{A number-theoretic identity}, Publ. Math. Debrecen 21 (1974), 67--69.


%\bibitem{Siva2} R. Sivaramakrishnan, \textit{Square reduced residue systems (mod $r$) and related arithmetical functions}, Canad. Math. Bull. 22 (1979), 207--220.


\bibitem{Sury} B. Sury, \textit{Some number-theoretic identities from group actions},
Rendiconti del Circolo Matematico di Palermo 58 (2009), 99--108.

\bibitem{tarn}
M. T$\breve{a}$rn$\breve{a}$uceanu, \textit{A generalization of Menon's identity}, J. Number Theory, 132 (2012), 2568--2573. \bibitem{Toth1} L. T\'oth, \textit{Menon-type identities concerning Dirichlet characters}, Int. J. Number Theory 14 (2018), 1047--1054.

\bibitem{Toth2}L. T\'oth, \textit{Menon's identity and arithmetical sums representing functions of several variables}, Rend. Semin. Mat. Univ. Politec. Torino 69 (2011), 97--110.

\bibitem{Z-C} Zhao, X.-P., Z-F. Cao,   \textit{Another generalization of Menon's identity},
Int. J. Number Theory 13 (2017), no. 9, 2373–-2379.
\end{thebibliography}
\end{document}